\documentclass[english]{article}

\title{The fundamental group of partial compactifications of the complement of a real line arrangement}
\author{Rodolfo Aguilar Aguilar\thanks{Partially supported by ERC ALKAGE and ANR project Hodgefun.}}

\date{}
\usepackage[english]{babel}
\usepackage{amsthm}
\usepackage{mathrsfs}
\usepackage{amsmath}
\usepackage{amssymb}

\usepackage{amsfonts}
\usepackage{enumitem}

\usepackage{graphicx}
\graphicspath{ {C:\Users\xckyu\OneDrive\Documentos\IF} }
\usepackage{tikz-cd}
\usepackage{tikz}
\usetikzlibrary{calc} 
\usetikzlibrary{arrows.meta,bending,decorations.markings,intersections} 

\tikzset{
    arc arrow/.style args={%
    to pos #1 with length #2}{
    decoration={
        markings,
         mark=at position 0 with {\pgfextra{%
         \pgfmathsetmacro{\tmpArrowTime}{#2/(\pgfdecoratedpathlength)}
         \xdef\tmpArrowTime{\tmpArrowTime}}},
        mark=at position {#1-\tmpArrowTime} with {\coordinate(@1);},
        mark=at position {#1-2*\tmpArrowTime/3} with {\coordinate(@2);},
        mark=at position {#1-\tmpArrowTime/3} with {\coordinate(@3);},
        mark=at position {#1} with {\coordinate(@4);
        \draw[-{Stealth[length=#2,bend]}]       
        (@1) .. controls (@2) and (@3) .. (@4);},
        },
     postaction=decorate,
     }
}

\theoremstyle{plain} 

\newtheorem{thm}{Theorem}[section]
\newtheorem{lem}[thm]{Lemma}
\newtheorem{prop}[thm]{Proposition}
\newtheorem{cor}{Corollary}
\newtheorem{defn}{Definition}[section]

\newtheorem{rem}{Remark}
\newtheorem{ques}{Question}

\DeclareMathOperator{\Bl}{Bl}

\DeclareMathOperator{\GL}{GL}
\DeclareMathOperator{\Sing}{Sing}

\newcommand{\abs}[1]{\left\vert#1\right\vert}
\newcommand{\A}{\mathscr{A}}

\begin{document}

\maketitle

\begin{abstract}
  Let $\A$ be a real projective line arrangement and $M(\A)$ its complement in $\mathbb{CP}^2$. We obtain an explicit expression in terms of Randell's generators of the meridians around the exceptional divisors in the blow-up $\bar{X}$ of $\mathbb{CP}^2$ in the singular points of $\A$. We use this to investigate the partial compactifications of $M(\A)$ contained in $\bar{X}$ and give a counterexample to a statement suggested by A. Dimca and P. Eyssidieux to the effect that the fundamental group of such an algebraic variety is finite whenever its abelianization is.
\end{abstract}

\section{Introduction}

The study of the fundamental group of smooth algebraic varieties is a classical problem in complex geometry. One of the most studied case is the complement of an arrangement of lines $\A$ in $\mathbb{P}^2$.  Several methods have been used for computing this group, for example: \cite{suciu}, \cite{salvetti}, \cite{Randell} . 

Consider a real arrangement of lines $\A=\{L_1,\ldots, L_n\}$ in $\mathbb{P}^2$ and denote by $\bar{X}$ the blow up at $\Sing \A=\{p_1,\ldots, p_s\}$, the set of multiple points of the arrangement . Denote by $D_i$ the strict transform of the lines in $\A$ and by $D_j$ the exceptional divisors. Let $D=\sum_{l=1}^{n+s}r_lD_l$ be a divisor in $\bar{X}$ with $r_l\in \mathbb{N}^*$ or equal to infinity. Denote by $r=(r_1,\ldots, r_s)$. To this datum we can associate the orbifold fundamental group $\pi_1(\mathcal{X}(X,D,r))$ (see e.g. \cite{Eyssidieux}).
\begin{thm}[Thm \ref{MainT}] \label{Thm1}
There is a presentation of $\pi_1(\mathcal{X}(\bar{X},D,r))$ obtained by adding to Randell's presentation relations that are powers of \underline{explicit} words in Randell's generators.
\end{thm}
The explicit algorithm to produce these words follows from a modification of Randell's method and is given in section \ref{3}.
The special case with $r_l= 1, \infty$ can be seen as a quasi-projective surface $X$ where the divisors with coefficients equal to infinity are removed from $\bar{X}$. This is, $X$ is a partial compactification of $\bar{X}\setminus D$ by those $D_i$ with coefficient $a_i=1$ (Linear Arrangement Compactifications or LAC surfaces in what follows). LAC surfaces are our main object of study in Section \ref{sec:Lac} and \ref{sec:App}. We show in Section \ref{sec:Lac} that it suffices to give weight one only to exceptional divisors in order to obtain different groups others than those who can be obtained from an arrangement of less lines.

We can ask if the following condition is satisfied by $X$ a quasi-projective variety
\begin{itemize}
\item If $\#\pi_1(X)=+\infty$, there is a representation $\rho:\pi_1(X)\to \GL_N(\mathbb{C})$ with $N\in \mathbb{N}^*$, such that $\#\rho(\pi_1(X))=+\infty$. (See (\cite{Ey2}) for motivation and related questions for Kähler groups). 
\end{itemize} 
No counterexample seems to be known. We give a negative answer in the case $N=1$ with $X$ a LAC surface in Theorem \ref{Counterexample}.
\begin{thm}[Thm \ref{Counterexample}]\label{Thm2} There exists a real arrangement of $6$ lines $\mathcal{B}$ (the complete quadrilateral) and a partial compactification $X$ of $\bar{X}\setminus D$ such that $\#\pi_1(X)=+\infty$ and $\#H_1(X)=4$.
\end{thm}
 To prove Theorem \ref{Thm2} we use Theorem \ref{Thm1} and obtain $ \pi_1(X)=\mathbb{Z}/2\mathbb{Z}\ast \mathbb{Z}/2\mathbb{Z}$ which can be faithfully embedded in $\GL_2(\mathbb{C})$ hence satisfis the above condition with $N=2$.
This group can be seen to be induced by a regular map $f$ from $X$ to $\mathbb{P}^1$ minus three points, coming from a pencil of conics and having two double fibers, in fact $f_*:\pi_1(X)\to \pi_1(\mathcal{X}(\mathbb{P}^2,D,(2,+\infty,2))$ where $D=(1:0)+(1:-1)+(0:1)$, is an isomorphism.

At the end of section \ref{sec:App} we give, as another application of Theorem \ref{Thm1}, a presentation of some weighted LAC surfaces which are among the quotients of the ball by a uniform lattice in $PU(2,1)$ considered in  \cite{DM}. This method for obtaining the presentation was not found by the author in the literature. 



\section{Preliminaries}
We review the definitions and some properties of meridians and orbifolds. For the latter we follow \cite{Eyssidieux}.
\subsection{Meridians}
Let $M$ be a connected complex manifold, $H\subset M$ a hypersurface, $D$ an irreducible component of $H$ and $q\in M\setminus H$. Denote by $U=\{z\in \mathbb{C}\mid \abs{z}<2\}$ and let $f:U\to M$ be a holomorphic function such that:
\begin{enumerate}
\item $f^{-1}(H)=\{0\}$,
\item $f(0)=p$ is an smooth point of $H$ and $p\in D$,
\item $f'(0)\not\in T_p H$ where $T_p H$ is the tangent space of $H$ at $p$.
\end{enumerate}
Then $f\mid_{S^1}:S^1\to M\setminus H$ defines a free-homotopy class independent of $f$ where $S^1\subset U$ is the unit circle. A loop $\gamma\in \pi_1(M\setminus H,q)$ freely homotopic to $f|_{S^1}$ is called a \emph{meridian} of $D$ around $p$.

If $D$ is smooth, any other meridian of $D$ around a smooth point of $H$ is a conjugate of $\gamma$. Denoting by $H'=H\setminus D$, we have that the inclusion $i:M\setminus H \hookrightarrow M\setminus H'$ induces a morphism $i_*:\pi_1(M\setminus H,q)\to \pi_1(M\setminus H',q)$ whose kernel is the normal subgroup of $\pi_1(M\setminus H,q)$ generated by $\gamma$. By Van Kampen's theorem the normal subgroup generated by the set of meridians around each irreducible component of $H$ is the kernel of the map $\pi_1(M\setminus H,q)\to \pi_1(M,q)$ induced by the natural inclusion.

Suppose  $H=D$ is smooth and let $\gamma_D$ be a meridian. 
Denote by $\pi: \bar{M} \to M$  the blow up of $M$ at some $p\in D$ and let $E_p$ be the exceptional divisor.  Then $\pi^{-1}(\gamma_D)$ is a meridian of $E_p$ in $\bar{M}$.

\subsection{Orbifolds}

Let $M$ be a complex manifold and $D$ a smooth effective divisor. Let $r\in \mathbb{N}^*$ and consider $P\to M$ the principal $\mathbb{C}^*$-bundle attached to $\mathscr{O}_M(-D)$. The tautological section $\sigma_D\in H^0(M,\mathscr{O}_M(D))$ can be lifted to a holomorphic function $f_D:P\to\mathbb{C}$ satisfying $f_D(p\cdot \lambda)=\lambda f_D(p)$. Let $Y\subset P\times \mathbb{C}$ be the complex analytic space defined by the equation $z^r=f_D(p)$ where $z$ is a coordinate for $\mathbb{C}$. Since $D$ is smooth $Y$ is smooth too. The action of $\mathbb{C}^*$ can be extended to $Y$ in the following way: $(p,z)\cdot\lambda=(p\cdot \lambda^r,\lambda z)$. Then the complex analytic stack
$$M(\sqrt[r]{D}):=[Y_D/\mathbb{C}^*]$$ is an orbifold.
The non-trivial isotropy groups lie over the points in $D$ and are isomorphic to the group $\mu_r$ of $r$-roots of unity. 

We allow also the weight $+\infty$ by considering the manifold $M\setminus D$ as an stack $[M\setminus D]$ and write
$$M(\sqrt[+\infty]{D}):=[M\setminus D]. $$

Let $\bar{X}$ be a complex manifold and $D=\sum_{i=1}^l D_i$ be a simple normal crossing divisor, where $D_i$ is an irreducible component of $D$. For any choice of weights $r:=(r_1,\ldots,r_l)\in (\mathbb{N}^*\cup \{+\infty\})^l$ we can define the orbifold
$$\mathcal{X}(\bar{X},D,r):=\bar{X}(\sqrt[r_1]{D_1})\times_{\bar{X}}\cdots\times_{\bar{X}}\bar{X}(\sqrt[r_l]{D_l})$$
 Denoting by $X=\bar{X}\setminus D$, we can view $\mathcal{X}(\bar{X},D,r)$ as an orbifold (partial if some $r_i=+\infty$) compactification of $X$. Let $j_r:X\hookrightarrow \mathcal{X}(\bar{X},D,r)$ denote the natural open immersion. By fixing $q\in X$, it turns out that we can define $\pi_1(\mathcal{X}(\bar{X},D,r),q)$ and moreover it is the quotient of $\pi_1(X,q)$ by the normal subgroup generated by all $\gamma_i^{r_i}$, where $\gamma_i$ is a meridian around $D_i$ and $r_i\not = +\infty$. We obtain that ${j_r}_*:\pi_1(X,q)\to \pi_1(\mathcal{X}(\bar{X},D,r),q) $ is surjective.   As a particular case we have that if $r=(1,\ldots,1)$ then $\mathcal{X}(\bar{X},D,r)=\bar{X}$.

 Let $D_\infty:= \sum  D_j$ the sum of all irreducible component of $D$ such that $r_j=+\infty$. We can regard $\mathcal{X}(\bar{X},D,r)$ as $\mathcal{X}(\bar{X}\setminus D_\infty, D-D_\infty,r')$ where $r'$ consists of the same finite values that $r$. In particular, if $r'_i=1$ for all $i$ we have that $\mathcal{X}(\bar{X},D,r)=[\bar{X}\setminus D_\infty]$ and we write simply $\bar{X}\setminus D_\infty$.
\begin{defn}
Let $X$ be a smooth algebraic variety, $Y$ a projective curve, $D=\sum_{i=1}^l y_i$ a divisor on $Y$ and $r\in (\mathbb{N}^*)^l$. Consider the orbifold $\mathcal{X}(Y,D,r)$. A dominant algebraic morphism $f:X\to Y$ is said to be an \emph{orbifold morphism} if for all $y_i \in D$ the multiplicity of the fiber $f^*(y_i)$ is divisible by $r_i$.
\end{defn}


\section{Fundamental group}\label{sec:FundG}
\subsection{Modification of the method of Randell}\label{subsec:2.1}
\subsubsection{Elementary geometric bases}
Consider $n$ real points $\{x_1,x_2,\ldots,x_n\}\subset \mathbb{R}\subset \mathbb{C}$ such that $x_1<x_2<\ldots <x_n$. Fix $q\in \mathbb{R}\setminus \{x_1,\ldots, x_n\}$. Any oriented simple closed curve $C\subset\mathbb{C}\setminus \{x_1,\ldots,x_n\}$  is freely homotopic to a loop based at $q$.  Moreover, if it contains at least one $x_i$ in the bounded component that $C$ determines, there exists a simple path $\theta$ connecting $q$ and $C$ satisfying:
$$ \Im(\theta(t))<0 \text{ for } t\in (0,1). $$
If $C\cap \{\Im(z)<0\}$ is connected we call $C_q:=\theta \cdot C \cdot \theta^{-1}$ an \emph{elementary loop}. Here  $\Im$ denotes the imaginary part of a complex number. (We suppose the curve starts at a point with $\Im(z)\leq 0$).
\begin{rem}
We have made all the choices in order to have $C_q$ unique in $\pi_1(\mathbb{C}\setminus \{x_1,\ldots,x_n\},q)$.\end{rem}
\begin{figure}[ht]
\centering
\begin{tikzpicture}
\draw (0,0) node [right]{$q$};
\draw (.5,0) node [below]{$x_2$};
\draw (-.5,0) node [below] {$x_1$};
\draw (0,0) to[out=-90,in=90]  (0,-.60) ;
\draw [arc arrow=to pos 0.35 with length 2mm]  (-1,0) to[out=-90,in=-90] 
(1,0) [arc arrow=to pos 0.85 with length 2mm] 
to[out=90,in=90] node[above] {$C$} cycle;

\foreach \Point in {(0,0),(-.5,0), (0.5,0)}{
    \node at \Point {\textbullet};
}

\end{tikzpicture}
\label{elementaryloop}
\caption{Elementary loop $C_q$.}
\end{figure}
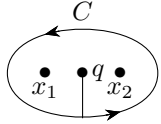
\begin{defn}
 An (ordered) \emph{geometric base} $\Gamma=(\gamma_1,\ldots,\gamma_n)$ for the group $\pi_1(\mathbb{C}\setminus \{x_1,\ldots,x_n\},q)$ is an $n$-tuple such that $\gamma_i$ is a meridian of $x_i$ based at $q$ and satisfying:
$$\gamma_n\cdot \gamma_{n-1}\cdots \gamma_{1}= \partial B(0,M)_q  $$
in $\pi_1(\mathbb{C}\setminus \{x_1,\ldots,x_n\},q)$, with $M>\abs{x_i}$ for all $i=1,\ldots,n$. The curve $\partial B(0,M)$ is a circle centered at $0$ with radius $M$ and oriented counterclockwise. We consider the product of loops from left to right.
\end{defn}

\begin{rem} The loop $\partial B(0,M)_q$ can be seen as the inverse of a meridian loop around the point at infinity.
\end{rem}
\noindent By abuse of notation we will write $\Gamma\subset \mathbb{C}$.
\begin{defn} An \emph{elementary geometric base} $\Gamma=(\gamma_1,\ldots,\gamma_n)$ is a geometric base  such that every $\gamma_i$ is an elementary loop.
\end{defn}
\begin{figure}[h]
\centering
\begin{tikzpicture} 
\draw (0,0) node [left]{$q$};
\draw (1.5,0) node [right]{$x_1$};
\draw (3,0) node [right]{$x_2$};
\draw (0,0) to[out=-90,in=-90] node[above]{$\gamma_1$}  (1,0) ;
\draw [arc arrow=to pos 0.35 with length 2mm] (1,0) to[out=-90,in=-90] 
(2,0) [arc arrow=to pos 0.85 with length 2mm] 
to[out=90,in=90] cycle;
\draw (0,0) to[out=-90,in=-90]node[below]{$\gamma_2 $}  (2.5,0) ;
\draw [arc arrow=to pos 0.35 with length 2mm] (2.5,0) to[out=-90,in=-90] 
(3.5,0) [arc arrow=to pos 0.85 with length 2mm] 
to[out=90,in=90] cycle;

\foreach \Point in {(0,0),(1.5,0), (3,0)}{
    \node at \Point {\textbullet};
}

\end{tikzpicture}
\label{elgeobas}
\caption{An elementary geometric base.}
\end{figure}
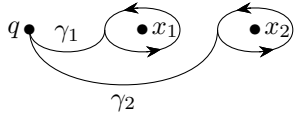
\begin{lem} Given $n$ real points and a base point as above, there is a unique \emph{elementary geometric base} $\Gamma$.
\end{lem}
\begin{proof}
Is immediate by the ordering of $\Gamma$ and the uniqueness of the elementary loops.
\end{proof}
\begin{rem} The notion of geometric base for $\pi_1((L\otimes \mathbb{C})\setminus P; q)$ depends only on the real oriented line $L$ and $P=\{x_1,\ldots,x_n\}\in L(\mathbb{R}), q\in L(\mathbb{R})$.
\end{rem}
\subsubsection{Randell's pencil} \label{randellspencil}
\begin{defn} A \emph{complex arrangement of lines} is an algebraic set $\A\subset \mathbb{P}^2$ whose irreducible components are complex lines. The arrangement $\A$ is said to be \emph{real} or \emph{to be defined over the reals} if the coefficients of all linear forms defining each line can be taken to be real.
\end{defn}
Denote by $M(\A):=\mathbb{P}^2\setminus \A$. We are going to review and adapt a method to compute a presentation for $\pi_1(M(\A))$ when $\A$ is real as in \cite{Randell}.

 Associate to each (projective) arrangement $\A$ an affine one, defined as fo\-llows: fix a line $L_\infty\in \A$ and consider it as a line at infinity, then $$\A^{\text{aff}}:=\A\cap (\mathbb{P}^2\setminus L_\infty)\cong\A\cap\mathbb{C}^2,$$ where we have chosen an isomorphism $h:\mathbb{C}^2\to\mathbb{P}^2\setminus L_\infty$. If we denote $M(\A^{\text{aff}}):=\mathbb{C}^2\setminus \A^{\text{aff}}$, we have the identification:
$$M(\A)= M(\A^{\text{aff}}).$$
Fixing $q\in M(\A^{\text{aff}})$ and denoting also by $q=h(q)$, we have:
$$\pi_1(M(\A),q)\cong \pi_1(M(\A^{\text{aff}}),q) .$$
Moreover if the arrangement $\A$ is real, we can associate it a planar graph (allowing rays) in $\mathbb{R}^2$. Suppose $\A^{\text{aff}}$ is the associated affine arrangement, then all multiple points lie in a real plane. Namely, if we consider $\mathbb{C}^2$ with coordinates $(z,w)=(x_1+iy_1,x_2+iy_2)$, the real plane is given by  $\{(z,w)\in \mathbb{C} \mid y_1=y_2=0\}\cong \mathbb{R}^2$. Set $\A(\mathbb{R}):=\A^{\text{aff}}\cap \mathbb{R}^2$ to be the set of real points of the arrangement $\A^{\text{aff}}$, denote by $M(\A(\mathbb{R})):=\mathbb{R}^2\setminus \A(\mathbb{R})$. Suppose there is no vertical line in $\A(\mathbb{R})$. Denote by $\Sing\A ^\bullet $ the multiple points of the corresponding arrangement $\A^\bullet=\A,\A^{\text{aff}},\A(\mathbb{R})$. 


 Consider $\mathbb{R}^2$ with coordinates $(x_1,x_2)$. We orient the non vertical  lines in $\mathbb{R}^2$ taking the positive direction to be that of decreasing $x_1$. 

Fix a base point $q=(q_1,q_2)$ in the lower right part of $M(\A(\mathbb{R}))$ further and lower than any point in $\Sing \A(\mathbb{R})$ and lower than any line. For a complex line $\Sigma\subset \mathbb{C}^2$ defined by an equation with real coefficients, denote by $\Sigma(\mathbb{R})$ its restriction to $\mathbb{R}^2$ and orient it as before if it is non-vertical.  Set $\Sigma^{(0)}:=\{(z,w) \mid z=q_1\}$, note that $\Sigma^{(0)}(\mathbb{R})$ is the vertical line passing through $q$, we orient it by taking as positive direction that of increasing $x_2$. For any triple $P\subset \Sigma(\mathbb{R}) \subset \Sigma$, where $P$ is a finite set of points, $\Sigma(\mathbb{R})$ a real oriented line and $\Sigma$ a complex line as before, we can consider an elementary geometric base $\Gamma\subset \Sigma$ of $\pi_1(\Sigma\setminus P,q)$ by fixing $q\in\Sigma(\mathbb{R})$.

 As $\Sigma^{(0)}(\mathbb{R})$ intersects all the lines of $\A(\mathbb{R})$, we can number $P=\Sigma^{(0)}(\mathbb{R})\cap \A(\mathbb{R})$ from bottom to top (given by the orientation chosen for $\Sigma^{(0)}(\mathbb{R}))$ and denote $\Gamma^{(0)}=\{\gamma_1^{(0)},\ldots, \gamma_n^{(0)}\} \subset \Sigma^{(0)}$ the associated elementary geometric base with base point $q$.

The idea to obtain a presentation for the fundamental group is to study how the elementary geometric base change when we rotate the line $\Sigma^{(0)}$ counterclockwise while fixing the base point $q$ and keep track of the relations arising. 

\begin{figure}[ht]
\centering
\begin{tikzpicture}[scale=1.5]
\draw (.5,-.4) node [below] {\scriptsize $q$};

\node at (-.11,.15){\tiny $p_1 $};
\node at (-.75,.28){\tiny $p_2 $};
\node at (-1.5,0){\tiny $p_3 $};
\node at (-.80,-.45){\tiny $p_4 $};
\draw (.05,1)node[above,right]{\scriptsize $3$} -- (-.05,-1);
\draw (-1,-3/4) -- (.5,3/8) node[above,right]{\scriptsize $2$};
\draw (-1,1/4) -- (.5,-1/6)node[above,right]{\scriptsize $1$};
\draw (-.55,1)node[above,right]{\scriptsize $4$} -- (-.65,-1);

\draw [line width=0.05mm,red ] (-.65,1) node [above]{\tiny $\Sigma^{(0)}$} -- (.6,-.8);
\draw [line width=0.05mm,red ] (-1,.55)node [above]{\tiny {$\Sigma^{(1)}$}} -- (.6,-.8);
\draw [line width=0.05mm,red ] (-1,.25)node [above,left]{\tiny {$\Sigma^{(2)}$}} -- (.6,-.8);
\draw [line width=0.05mm,red ] (-1,-.35)node [left]{\tiny {$\Sigma^{(3)}$}} -- (.6,-.8);
\foreach \Point in {(.6,-.8),(0,0),(-.6,.15),(-.6,-.45)}{
    \node at \Point {$\cdot$};
}

\draw [red] (-.4,.6) to[out=-150,in=100] 
(-.5,.3) [arc arrow=to pos 1 with length .5mm] ;

\draw [red] (-.7,0) to[out=-150,in=100] 
(-.8,-.3) [arc arrow=to pos 1 with length .5mm] ;
\end{tikzpicture}
\caption{Base point}
\label{fig:Base Point}
\end{figure}
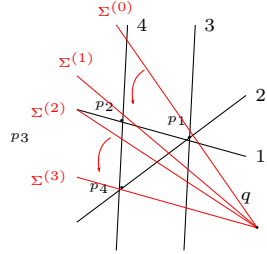

The set of lines passing through $q$, can be seen as $\mathbb{RP}^1$, which we parametrised by the angle with respect to the line $x_2=0$ (oriented in the positive sense), this is, a value in $[\pi/2,3\pi/2[$. To every real line $\Sigma(\mathbb{R})$ passing through $q$ we can associate its angle, which we denote by:
$$\theta(\Sigma(\mathbb{R}))\in [\pi/2,3\pi/2[. $$
For $t\in [\pi/2,3\pi/2[$, the line being parametrised by $t$ will be denoted by $\Sigma_t$.

In particular,  $\theta(\Sigma^{(0)})=\pi/2$. The elementary geometric base $\Gamma^{(0)}$ varies in a continuous way as we vary $t$.
There exists two types of directions where it changes:

\begin{enumerate}[label=\textbf{S.\arabic*}]
\item \label{en1} Those $t\in[\pi/2,3\pi/2[$ such that the associated $\Sigma_t$ contains a point in $\Sing \A(\mathbb{R})$,
\item \label{en2} Those $t\in[\pi/2,3\pi/2[$ such that $\Sigma_t$ is parallel to a line in $\A(\mathbb{R})$, which correspond to the points in $\Sing \A\cap L_\infty$.
\end{enumerate}

By a slight change of $q$, we can consider that no line passing through it contains two points of $\Sing \A$. Given $p\in\Sing\A$, denote by $\theta(p)$ the angle of the unique line passing through $p$ and $q$. Given $p, p'\in \Sing \A$, we define a total order by 
$$p<p' \quad \text{iff} \quad \theta(p)<\theta(p'). $$ Let us write $\Sing \A=\{p_1,\ldots, p_s\}$ with this order.

\subsubsection{Elementary geometric transition of regular fibers in Randell's pencil}\label{3.1.3}
 Fix a point $p_i\in\Sing \A$. Denote by $t_i=\theta(p_i)$. Choose $\varepsilon>0$ sufficiently small such that no $t\in [t_i-\varepsilon, t_i+\varepsilon]\setminus \{t_i\}$ is of type \ref{en1} or \ref{en2}. Let:
$$\Sigma^{(i-1)}:= \Sigma_{t_i-\varepsilon}, \quad  \Sigma^{(i)}:=\Sigma_{t_i+\varepsilon}.$$
This is, $\Sigma^{(i-1)}$ lies to the right and $\Sigma^{(i)}$ to the left of $p_i$. Recall that $\Sigma^{(i-1)}(\mathbb{R})$ is an oriented real line and by intersecting with $\A(\mathbb{R})$ we can consider the elementary geometric base: $$\Gamma^{(i-1)}=(\gamma_1^{(i-1)},\gamma_2^{(i-1)},\ldots,\gamma_n^{(i-1)})\subset \Sigma^{(i-1)},$$ similarly $$\Gamma^{(i)}=( \gamma_1^{(i)}, \gamma_2^{(i)}, \ldots, \gamma_n^{(i)}) \subset \Sigma^{(i)}.$$ 

A priori, we should take such geometric bases for every point $p_i$ but as there is no direction between $t_i$ and $t_{i+1}$ in which the geometric base changes, by continuity we will still write  $\Gamma^{((i+1)-1)}={\Gamma^{(i)}}$. 

\begin{rem}
In fact, only the points of type $\ref{en1}$ play a role in the presentation of $\pi_1(M(\A))$. The points of type \ref{en2} does not modify the meridians who are about to cross a point in $\Sing \A(\mathbb{R})$, they only change their numeration in the geometric base. These points are studied in section \ref{subsection2.2} and they are needed for the explicit form of the exceptional meridians given in Section \ref{3}.
\end{rem}

 As a simple example illustrating how $\Gamma^{(i-1)}$ and $\Gamma^{(i)}$ are related, we have the following Lemma, which can be found in \cite{Orlik} Lemma 5.73 (with other notation). Let $\A^{\text{aff}}=\{L_1,L_2\}$ be a pencil of two lines in $\mathbb{C}^2$ defined over $\mathbb{R}$ and $x_i=\Sigma^{(0)}(\mathbb{R})\cap L_i(\mathbb{R})$ for $i=1,2$ as in \ref{randellspencil}. 
\begin{lem}\label{looppassing} Consider $\Gamma^{(0)}=(\gamma_1^{(0)},\gamma_2^{(0)})\subset \Sigma^{(0)}$ the elementary geometric base associated to $P=\{x_1,x_2\}$ and suppose $q<x_1<x_2$. Then $\Gamma^{(1)}=(\gamma_1^{(1)},\gamma_2^{(1)})$ can be represented in $\Sigma^{(0)}$ and is given by
$\Gamma^{(1)}=({\gamma_1^{(0)}}^{-1}\gamma_2^{(0)}{\gamma_1^{(0)}},\gamma_1^{(0)})$. (See Figure \ref{fig:loop}). 
\end{lem}
\begin{figure}[ht]
\centering
\begin{tikzpicture}
\draw (0,0) to[out=0,in=180]node[below]{$\gamma_1^{(0)}$}  (1,0) ;
\draw [arc arrow=to pos 0.25 with length 2mm] (1,0) to[out=-90,in=-90] 
(2.5,0) [arc arrow=to pos 0.75 with length 2mm] 
to[out=90,in=90] cycle;

\draw (0,0)node [below]{$q$};
\draw (1.75,0) node [below] {$x_1$};
\draw (3.75,0) node [below]{$x_2$};
\draw (0,0) to[out=-45,in=-135] node[midway,below]{$\gamma_2^{(0)}$} (3,0);
\draw (0,0) to[out=45,in=135]node [midway,above]{$\gamma_1^{(1)}$} (3,0) ;

\draw [arc arrow=to pos 0.25 with length 2mm] (3,0) to[out=-90,in=-90] 
 (4.5,0) [arc arrow=to pos 0.75 with length 2mm]  
to[out=90,in=90] cycle;

\foreach \Point in {(0,0),(1.75,0), (3.75,0)}{
    \node at \Point {\textbullet};
}

\end{tikzpicture}

\caption{Loop passing}
\label{fig:loop}
\end{figure}
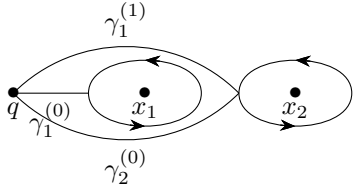
 We can interpret this Lemma as a loop passing a vertex and having a conjugation and a reordering of the lines as in Figure \ref{fig:conjugates}. 

The next example should be a pencil of $n$ lines, but in fact this is locally the general case as we will see in the following proposition. Let $p_i\in \Sing \A(\mathbb{R})$, $\Gamma^{(i-1)}$ and $\Gamma^{(i)}$ as before.
\begin{prop}\label{1eq} Let $j$ be the first index for which the meridian $\gamma_j^ {(i-1)}$ su\-rrounds a line which passes through $p_i$ and let $k$ be the last such index. Then we have:
$$\Gamma^{(i)}=(\gamma_1^{(i-1)},\ldots,\gamma_{j-1}^{(i-1)},{\gamma_j^{(i)}},\ldots,{\gamma_k^{(i)}},\gamma_{k+1}^{(i-1)},\ldots,\gamma_n^{(i-1)}), $$
where:
\begin{align*}
{\gamma_{k}^{(i)}}&=\gamma_{j}^{(i-1)}, \\
\gamma_{k-1}^{(i)}&={\gamma_{j}^{(i-1)}}^{-1}\gamma_{j+1}^{(i-1)}\gamma_{j}^{(i-1)},\\
&\ \ \!\vdots \\
\gamma_{j}^{(i)}&
={\gamma_{j}^{(i-1)}}^{-1}{\gamma_{j+1}^{(i-1)}}^{-1}\cdots{\gamma_{k-1}^{(i-1)}}^{-1}\gamma_{k}^{(i-1)}\gamma_{k-1}^{(i-1)}\ldots{\gamma_{j}^{(i-1)}}\gamma_{j}^{(i-1)}.
\end{align*}
And a set of relations in $\pi_1(M(\A),q)$:\footnote{These relations are stated as in \cite{Falk} p.142, where in a footnote he points to an error of \cite{Randell}.}
\begin{equation}\label{rpi}
R_{p_i}=\left\lbrace\gamma_{k}^{(i-1)}\gamma_{k-1}^{(i-1)}\cdots\gamma_{j}^{(i-1)} = \gamma_{\sigma(k)}^{(i-1)}\gamma_{\sigma(k-1)}^{(i-1)}\cdots\gamma_{\sigma(j)}^{(i-1)}=\ldots \right\rbrace
\end{equation}
where $\sigma$ runs over the set of cyclic permutations of $k-j+1$ elements.
\end{prop}

\begin{figure}[ht]
\centering
\begin{tikzpicture}
\draw (0,0) node [below] {$p_i$};
\draw (-1.5,1) node[left] {$\gamma_{k}^{(i)}$ } -- (1.5,-1) node[right] { $\gamma_{j}^{(i-1)}$};
\draw (-1.5,.3) node[left] {$\gamma_{k-1}^{(i)}$   } -- (1.5,-0.3) node[right] { $\gamma_{j+1}^{(i-1)}$};

\draw (-1.5,-.3)node[left] {$\vdots$ } --(1.5,0.3) node[right] { $\vdots$};
\draw (-1.5,-1)node[left] {$\gamma_{j}^{(i)}$ } --(1.5,1) node[right] { $\gamma_{k}^{(i-1)}$};
\draw [red](-1.5,1.2)node[above]{$\Sigma^{(i)} $} -- (-1.5,-1.2);
\draw [red](1.5,1.2)node[above] {$\Sigma^{(i-1)} $} -- (1.5,-1.2);
\end{tikzpicture}
\caption{Conjugates} 
\label{fig:conjugates}
\end{figure}
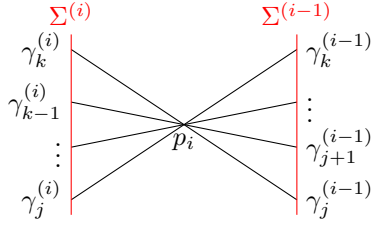

Expressing every meridian in terms of the geometric base $\Gamma^{(0)}$ by means of  Proposition \ref{1eq} and replacing in (\ref{rpi}) we obtain:
\begin{thm}[\cite{Randell}]\label{presentationRandell} The fundamental group of 
$M(\A)$ admits a presentation:
$$\pi_1(M(\A),q)\cong\left\langle \gamma_1^{(0)},\gamma_2^{(0)},\ldots, \gamma_n^{(0)}\biggm| \bigcup_ {p_i \in\Sing(\A(\mathbb{R}))}R_{p_i} \right\rangle. $$
\end{thm}
\subsection{Meridians crossing a point at infinity}\label{subsection2.2}
 Let us describe the change in the geometric base when it traverses a singular point at infinity. Let $p_i\in \Sing \A\cap L_\infty$ and $\Gamma^{(i-1)}\subset \Sigma^{(i-1)}$ and $\Gamma^{(i)}\subset \Sigma^{(i)}$ be given as in section \ref{3.1.3}. This is: $$\Gamma^{(i-1)}=( \gamma_1^{(i-1)},\ldots, {\gamma_{n-k+1}^{(i-1)}}, \ldots, {\gamma_n^{(i-1)}}),$$ and $$\Gamma^{(i)}= (\gamma_1^{(i)},\ldots, \gamma_k^{(i)},\gamma_{k+1}^{(i)},\ldots, \gamma_n^{(i)}).$$
 \begin{prop}\label{Prop:3.5}
   Assume that there are exactly $k$ parallel lines in $\A(\mathbb{R})$ whose corresponding lines in $\A$ intersect at $p_i$. Then these lines are associated to the last $k$ meridians ${\gamma_{n-k+1}^{(i-1)}}, \ldots, {\gamma_n^{(i-1)}}$ of $\Gamma^{(i-1)}$.
   \end{prop}
   \begin{proof}
Let $t_i=\theta(p_i)$ and $\Sigma_{t_i}$ be the line passing by $q$ and $p_i$. Using the order of the real lines write $\Sigma^{(i-1)}(\mathbb{R})\cap \A(\mathbb{R})=\{y_1,\ldots,y_n\}$ and as no other point of $\Sing \A$ different from $p_i$ lies in $\Sigma_{t_i}$ we have that $\Sigma_{t_i}\cap \A(\mathbb{R})=\{x_1,\ldots,x_{n-k}\}$. In fact, it must be the case that $x_i$ and $y_i$ are in the same line of $\A(\mathbb{R})$ otherwise a point of type \ref{en1} or \ref{en2} would lie between $\Sigma^{(i-1)}$ and $\Sigma_{t_i}$ which can not happen.
   \end{proof}
  \begin{cor} We have the following identifications in $\Sigma_{t_i}$:
 \begin{align}\label{eq:gbinf1}
 \gamma_{k+1}^{(i)}= {\gamma_1^{(i-1)}}, \quad \ldots \quad ,\gamma_n^{(i)}={\gamma_{n-k}^{(i-1)}}.
 \end{align}
  \end{cor}
  \begin{proof}
As we are turning counter-clockwise, by the orientation given to $\Sigma^{(i)}(\mathbb{R})$ it will intersect first the $k$ parallel lines associated to $p_i$ and then, by the same argument as in Proposition \ref{Prop:3.5}, the point in the position $k+j$ of $\Sigma^{(i)}(\mathbb{R})\cap \A(\mathbb{R})$ lies in the same line as $x_j$ . 
  \end{proof}
\begin{prop} The last $k$ meridians in $\Gamma^{(i-1)}$ invert their order to fit in the first $k$ places of $\Gamma^{(i)}$. By doing so a conjugation for all the precedent meridians is needed (see figure \ref{baseinfty}).  More precisely we have:
\begin{align}\label{eq:gbinf2}
\gamma_{k}^{(i)}&={\gamma_1^{(i-1)}}^{-1}\cdots {\gamma_{n-k}^{(i-1)}}^{-1}\ {\gamma_{n-k+1}^{(i-1)}}\ {\gamma_{n-k}^{(i-1)}}\cdots {\gamma_1^{(i-1)}},\nonumber \\
\gamma_{k-1}^{(i)}&={\gamma_1^{(i-1)}}^{-1}\cdots {\gamma_{n-k+1}^{(i-1)}}^{-1}\ {\gamma_{n-k+2}^{(i-1)}}\ {\gamma_{n-k+1}^{(i-1)}}\cdots {\gamma_1^{(i-1)}},\nonumber \\
&\ \ \!  \vdots \nonumber\\
\gamma_{1}^{(i)}&={\gamma_1^{(i-1)}}^{-1}\cdots {\gamma_{n-1}^{(i-1)}}^{-1}\ {}\gamma_{n}^{(i-1)}\ {\gamma_{n-1}^{(i-1)}}\cdots {\gamma_1^{(i-1)}}.
\end{align}
\end{prop}
\begin{proof}
 By repeated iterations of Lemma \ref{looppassing} we obtain (\ref{eq:gbinf2}). The result follows by unicity of the elementary geometric base.
 \end{proof} 
 \begin{figure}[ht]
\centering
\begin{tikzpicture}
\draw (0,0) to[out=0,in=180]node[above]{$^{(i)}\gamma_1$}  (1,0) ;
\draw [arc arrow=to pos 0.25 with length 2mm] (1,0) to[out=-90,in=-90] 
(2.5,0) [arc arrow=to pos 0.75 with length 2mm] 
to[out=90,in=90] cycle;

\draw (0,0)node [below]{$q$};
\draw (1.75,0) node [below] {$x_1$};
\draw (3.75,0) node [below]{$x_2$};
\draw (5.75,0) node [below]{$x_3$};
\draw (7.75,0) node [below]{$x_4$};
\draw (0,0) to[out=-45,in=-135] node[midway,below]{$^{(i)}\gamma_2$} (3,0);
\draw (0,0) to[out=45,in=135] node [midway,right]{$\gamma_2^{(i)}$}  (5,0) ;
\draw (0,0) to[out=45,in=135]node [midway,above]{$\gamma_1^{(i)}$} (7,0) ;

\draw [arc arrow=to pos 0.25 with length 2mm] (3,0) to[out=-90,in=-90] 
 (4.5,0) [arc arrow=to pos 0.75 with length 2mm]  
to[out=90,in=90] cycle;
\draw [arc arrow=to pos 0.25 with length 2mm] (5,0) to[out=-90,in=-90]  (6.5,0) [arc arrow=to pos 0.75 with length 2mm]  
to[out=90,in=90] cycle;
\draw [arc arrow=to pos 0.25 with length 2mm] (7,0) to[out=-90,in=-90] 
 (8.5,0) [arc arrow=to pos 0.75 with length 2mm]  
to[out=90,in=90] cycle;

\foreach \Point in {(0,0),(1.75,0), (3.75,0), (5.75,0), (7.75,0)}{
    \node at \Point {\textbullet};
}

\end{tikzpicture}

\caption{Loops crossing a point \ref{en2}.}
\label{baseinfty}
\label{loop}
\end{figure}
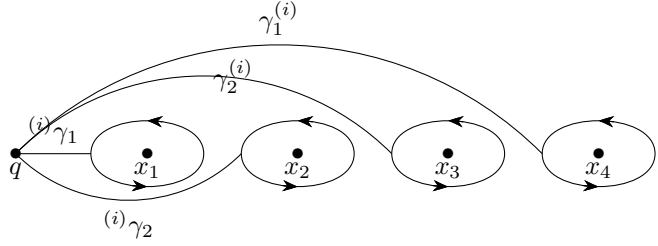
\subsection{Loops around singular points}\label{3}
Consider an arrangement $\A$ defined over the reals as in the precedent section. We have a canonical way of associating an elementary geometric base for every line $\Sigma_t$ passing though $q$ with $t\in[\pi/2,3\pi/2[$ as in $\ref{randellspencil}$.
We will write the elementary geometric base over the directions of the points \ref{en1} and \ref{en2} in terms of the elements of $\Gamma^{(i)}$.

 This can be seen as finding elementary loops for the points in $\Sing\A$, which can be divided into finite distance points $\Sing \A^{\text{aff}}\cong \Sing \A(\mathbb{R})$ and infinite distance $\Sing \A \cap L_\infty$.


\begin{lem}\label{lineatinfinity} The inverse of a  meridian loop around the line at infinity  $L_\infty$ at the point $L_\infty\cap \Sigma^{(i)}$ is given by the product of the elements of the elementary base $\Gamma^{(i)}$, this is :
$$(\gamma_{\infty}^{(i)})^{-1}=\gamma_n^{(i)}\cdot\gamma_{n-1}^{(i)}\cdots \gamma_1^{(i)}	 $$
\end{lem}
\begin{proof}
This is a simple consequence on the definition of geometric base and the choice of the base point. 
\end{proof}
 This meridian can be seen as an \emph{elementary} loop, as it is product of loops of this type.
 
 Recall that $t_i=\theta(p_i)$ denotes the angle of the line $\Sigma_{t_i}$ containing $p_i$ and $q$.
\begin{defn} A meridian $\gamma_{p_i}'$ \emph{around a singular point} $p_i\in\Sing \A(\mathbb{R})$, is a meridian of $p_i\in \Sigma_{t_i}$ (based at $q$).
\end{defn}

We can consider the \emph{elementary} meridian $\gamma_{p_i}$ as the elementary loop of $p_i$ in $\Sigma_{t_i}$ (based at $q$). With the notation of Proposition \ref{1eq} we have:

\begin{lem}\label{finite} The elementary meridian $\gamma_{p_i}$ can be obtained as a product of the elements of $\Gamma^{(i-1)}$ which surround the lines passing through $p_i$. Namely 
$$\gamma_{p_i}=\gamma_{k}^{(i-1)}\gamma_{k-1}^{(i-1)}\cdots\gamma_{j+1}^{(i-1)} \gamma_{j}^{(i-1)}.$$
\end{lem}
\begin{proof}
An elementary geometric base $\Gamma$ is constructed in such a way that the product of $k-j+1$ consecutive elements $(\gamma_{j},\ldots,\gamma_{k})$ of $\Gamma$ equals an elementary loop $C_q$, where $C$ is an oriented counter-clockwise simple closed curve that in the bounded part that it determines contains exactly $\{x_{j},\ldots,x_{k}\}$.
\end{proof}
\noindent Next we determine the meridians around multiple points lying in the line at infinity. 

Let $p_i\in \Sing\A\cap L_\infty$. Consider the line $\Sigma_{t_i}$ passing through $q$ and $p_i$. Suppose there are exactly $k$ lines in $\A$ different from $L_\infty$ passing through $p_i$ then their real points are parallel lines to $\Sigma_{t_i}(\mathbb{R})$ in $\A(\mathbb{R})$. As is Section \ref{subsection2.2} we have: $$\Sigma_{t_i}(\mathbb{R})\cap \A(\mathbb{R})=\{x_1,\ldots,x_{n-k}\}, $$
with $k\geq 1$ depending on $i$. The order of the points $x_i$ given by the orientation of $\Sigma_{t_i}(\mathbb{R})$. Hence we can take the elementary geometric base  $\Gamma_{t_i}\subset \Sigma_{t_i}$ associated with $P=\{x_1,\ldots,x_{n-k}\}$ and the base point $q$. Suppose $\Gamma_{t_i}=(\gamma_1,\ldots,\gamma_{n-k})$.
\begin{defn} A meridian $\gamma_{p_i}$ around a singular point at infinity $p_i\in \Sing \A\cap L_\infty$ is a meridian at infinity of $\Sigma_{t_i}$.
\end{defn}

\begin{lem}\label{meridian3} Let $\Gamma^{(i-1)}=(\gamma_1^{(i-1)},\ldots, \gamma_n^{(i-1)})$ be as in Section \ref{3.1.3}. For every point $p_i\in \Sing\A\cap L_\infty$  the elementary meridian $\gamma_{p_i}$ is given by any of the equivalent expressions
\begin{equation}\label{eq:6}
\gamma_{p_i}=\gamma_\infty^{(i-1)}\cdot\gamma_n^{(i-1)}\cdots\gamma_{n-k+1}^{(i-1)},
\end{equation}
or
\begin{equation}\label{eq:7}
\gamma_{p_i}^{-1}=\gamma_{n-k}^{(i-1)}\cdots\gamma_1^{(i-1)} .
\end{equation}

\end{lem}
\begin{rem}
In (\ref{eq:6}) a similitude with the formula of Lemma \ref{finite} can be observed. Namely the product of the meridians of the lines crossing the point $p_i$ give the meridian.
In (\ref{eq:7}) we simply compute the meridian around the point at infinity in the line $\Sigma_{t_i}$, so it is closer to Lemma \ref{lineatinfinity}.
\end{rem}
\begin{proof}


As no other point of $\Sing \A$ lies in $\Sigma_{t_i}$ by continuity, Proposition \ref{Prop:3.5} and the uniqueness of the elementary geometric base we have that $$\Gamma_{p_i}=(\gamma_1^{(i-1)},\ldots, \gamma_{n-k}^{(i-1)}),$$  
by applying Lemma \ref{lineatinfinity} we obtain (\ref{eq:7}).

In $\Sigma^{(i-1)}$ we have 
\begin{equation*}
\gamma_\infty^{(i-1)}=( \gamma_1^{(i-1)})^{-1}\cdots (\gamma_n^{(i-1)})^{-1},
\end{equation*}
therefore for the right hand side of (\ref{eq:6})
\begin{equation*}
\gamma_\infty^{(i-1)}\cdot\gamma_n^{(i-1)}\cdots\gamma_{n-k+1}^{(i-1)}=(\gamma_1^{(i-1)})^{-1}\cdots (\gamma_{n-k}^{(i-1)})^{-1}
\end{equation*}
 which equals $\gamma_{p_i}$ by (\ref{eq:7}).
\end{proof}
\begin{rem}
By the results of this Section we have obtained singular meridians for every point $p\in \Sing \A$ with $\A$ defined over the reals. In \cite{garber} Garber generalize a formula of Fujita \cite{fujita} expressing locally the singular meridians as the product of the meridians of the irreducible components in the singular point. He then uses this result globally when the lines intersect transversally, this is, when there is no additional conjugation. Our method can be seen as a generalization of this by allowing multiple points of higher order.
\end{rem}

\section{LAC Surfaces} \label{sec:Lac}

\subsection{Construction}\label{subsection4.2}
 We will construct surfaces generalizing the complement of a hyperplane arrangement and obtain a presentation of them.

Fix $\A=\{L_1,\ldots,L_k\}$ an arrangement of lines in $\mathbb{P}^2$. Let $\bar{X}$ be the blow up of $\mathbb{P}^2$ at  $\Sing \A=\{p_1,\ldots, p_s\}$ and  $\pi:\bar{X}{\to} \mathbb{P}^2$ the projection map. Denote by $D_1,\ldots, D_k$ the strict transform of the lines $L_1,\ldots, L_k$ and by  $D_{k+1},\ldots, D_{k+s}$ the exceptional divisors associated to the points $p_1,\ldots, p_s$.

Given a subset $I\subset \{1,\ldots, k+s\}$ we can define the orbifold $\mathcal{X}(\bar{X},D,r_I)$ associated to the divisor $D=\sum D_i$ and the weights $r_I=(r_1,\ldots, r_{k+s})$ where $r_i=1$ if $i \in I$ and $r_i=+\infty$ if $i\not\in I$. Then $\mathcal{X}(\bar{X},D,r_I)=\bar{X}\setminus (D_I)_\infty$ where we have written $D_I$ for $D$ to emphasize the dependence on $I$.
\begin{defn}
We call $\bar{X}\setminus (D_I)_\infty$ a  \emph{ (partial) Linear Arrangement Compactification} or LAC surface. 
\end{defn}
\begin{rem}\label{rema3}
If $I=\varnothing$, $(D_I)_\infty = D$ and $\pi$ restricted to $\bar{X}\setminus D$ is a biholomorphism with $M(\A)$, from which it follows that 
\begin{equation}\label{bihfunda}
\pi_1(\bar{X}\setminus D)\cong \pi_1(M(\A)),
\end{equation}
showing that these surfaces are indeed generalizations of the complement of an arrangement.
\end{rem}
\subsection{Reduced LAC Surfaces}
In \cite{Eyssidieux} a comment before Proposition 1.3 mentions that the log pair $(\bar{X},D)$ has to be rigid if one wants the fundamental group to be very different from $\bar{X}\setminus D$. We prove here that we can reduce the study of LAC surfaces to partially compactify only with respect to exceptional divisors, this is, the subset of irreducible components of $D$ with weight $1$ are exceptional divisors.

  We do so by showing that if a strict transform of a line $L_i$ has weight $1$, then we can find an arrangement of less lines whose associated LAC surface has the same fundamental group.
   In this process the double points lying in the line that we have removed create isolated points and we must allow to blow them up as well in order to cover the case when this exceptional divisor had weight $1$ in $\A$. With this is mind we have the following definition.

\begin{defn}
A \emph{LAC datum}, is a triple
$$(\A,S,I):= (\A=\{L_1,\ldots, L_k\}\subset \mathbb{P}^2, S=\{p_1,\ldots,p_s\}\subset \mathbb{P}^2, I\subset \{1,\ldots,k+s\}) $$
where $\A$ is an arrangement of lines in $\mathbb{P}^2$, $S$ a finite set of points and $I$ an index set.
\end{defn}

Given a LAC datum $(\A,S,I)$ we can construct a surface $\bar{X}\setminus (D_I)_\infty$ as in \ref{subsection4.2}. Consider $\bar{X}$ the blow up of $\mathbb{P}^2$ in the points $S$, call $D_1,\ldots, D_k$ the strict transform of the lines in $\A$, $D_{k+1},\ldots, D_{k+s}$ the exceptional divisors and  $(D_I)_\infty=\sum_{j\in J}D_j$ where $J:=\{1,\ldots, k+s\} \setminus I$. As we can change the arrangement and the set of points to blow up, we prefer the notation $M(\A,S,I)$ for this surface.

\begin{defn}
Two LAC datum $(\A,S,I)$, $(\A', S', I')$ are said to be equivalent if and only if 
$$\pi_1(M(\A,S,I))\cong \pi_1(M(\A',S',I')).$$ In such a case we write $(\A,S,I)\sim(\A', S', I')$.
\end{defn}

\begin{defn} A LAC datum $(\A,S, I)$ such that $S\subset \Sing \A$ and $I=S$ is called \emph{reduced}. In this case we write $(\A,I)$.  
\end{defn}

\begin{thm}\label{reducedLac} For every LAC $(\A, S, I)$ there is a canonical equivalent reduced LAC $(\A',I')$.
\end{thm}
\noindent We will need to prove first three reduction Lemmas.
\begin{lem}\label{RL1} Let $(\A, S,I)$ be a LAC datum. Suppose there exists $L_i\in \A$ such that $i\in I$, then 
$$(\A, S,I) \sim (\A\setminus {L}, S,I\setminus \{i\}) .$$
\end{lem}
\begin{proof} 
Denote by $\bar{X}=\Bl_S \mathbb{P}^2$. As $M(\A,S,I)=\bar{X}\setminus (D_I)_\infty$ and 
$$\{1,\ldots,s+k\}\setminus I = \{1,\ldots,\hat{i},\ldots,s+k\}\setminus (I\setminus \{i\}) $$ 
denoting by  $(D'_{I\setminus \{i\}})_\infty$ the divisor to be removed given by $(\A\setminus L,S, I\setminus \{i\})$ we have that 
$$(D_I)_\infty=(D'_{I\setminus \{i\}})_\infty $$
 that implies
$$M{(\A, S, I)}= M(\A\setminus L,S,I\setminus \{i\}). $$
\end{proof}

\noindent So we can suppose $I\subset \{1,\ldots, s\}$. The next step is to consider points lying outside $\A$. 
\begin{lem}\label{RL2}

Let $(\A,S,I)$ be a LAC datum  such that there is $p_j\in S$ that lies in no line of $\A$. 
\begin{enumerate}
\item If $j\in I$ then $$ (\A,S,I) \sim (\A,S\setminus \{p_j\}, I\setminus  \{j\}).$$

\item If $j\not\in I$ then 
$$ (\A,S,I) \sim (\A,S\setminus \{p_j\}, I).$$
\end{enumerate}
\end{lem}
\begin{proof}
\begin{enumerate}
\item The surface  $M(\A,S,I)$ 
 is the blowing up of $M(\A,S\setminus \{p_j\}, I\setminus \{j\}) $ at the point $p_j$, as the fundamental group is invariant under blow ups we obtain the stated.
\item 
We have a biholomorphism given by restricting the blowing up map of $M(\A, S\setminus\{p_j\},I)$ at $p_j$, to the complement of the exceptional divisor
$$ M(\A,S,I) \overset{\sim}{\to} M(\A, S\setminus\{p_j\},I)\setminus \{p_j\} $$
but as 
$$ \pi_1(M(\A, S\setminus\{p_j\},I)\setminus \{p_j\})\cong \pi_1(M(\A, S\setminus\{p_j\},I))$$
the result follows.
\end{enumerate}
\end{proof}

\noindent The last reduction Lemma, can be divided into two parts. In the first case we show that it is only interesting when we blow up a point and do not remove the exceptional divisor. In the second part, a point of $p_j\in S$ that is a smooth point of $\A$ does not affect the fundamental group in either case $j\in I$ or $j\not\in I$. By the last Lemma we can assume that every point in $S$ lies in the arrangement $\A$.
\begin{lem}\label{RL3}
Let $p_j\in S\subset \A$. 
\begin{enumerate}
\item If $j\not\in I$ then 
$$ (\A,S,I) \sim (\A,S\setminus \{p_j\}, I).$$
\item Suppose $p_j\in L$ for some line $L\in \A$. If $j\in I$ and $p_j\not\in \Sing \A$, then
$$ (\A,S,I) \sim (\A\setminus \{L\},S\setminus\{p_j\}, I\setminus \{j\}).$$
\end{enumerate}
\end{lem}
\begin{proof}
\begin{enumerate}
\item
Let $\bar{X}=\Bl_{S\setminus \{p_j\}}\mathbb{P}^2$ and $\bar{Y}=\Bl_{p_j} \bar{X}$. 
In $\bar{X}$ we have 
 $$(D_I)_\infty=\sum D_r \text{ with } r \in S\setminus (I\cup \{p_j\}) $$

In $\bar{Y}$
$$(D'_I)_\infty= (D_I)_\infty+ D'_j $$
Where we have denote also by $(D_I)_\infty$ the strict transform of the divisor with same notation in $\bar{X}$. Therefore we have a biholomorphism
$$\bar{Y}\setminus (D'_I)_\infty= M(\A, S, I) \overset{\sim}{\to} M(\A,S\setminus\{p_i\},I)\setminus\{p_j\}=\bar{X}\setminus ((D_I)_\infty \cup\{p_j\})$$
and the result follows.
\item  If $\gamma_{p_j}$ is a meridian at $p_j$ of $L$, as $p_j$ is a smooth point then it is also a meridian of the exceptional divisor $D_{k+j}$ in $M(\A,S,I\setminus \{j\})$. As $D_{k+j}$ is smooth, $\gamma_{p_j}$ generates the kernel of 
$$\pi_1(M(\A, S, I\setminus \{j\}))\to \pi_1(M(\A,S,I))$$
hence
\begin{equation}\label{eq:reduction}
\pi_1(M(\A, S, I\setminus \{j\}))/\left\langle\left\langle\gamma_{p_j} \right\rangle\right\rangle \cong \pi_1(M(\A,S,I)) 
\end{equation}
 By the point 1 above we have that 
$ (\A,S,I\setminus \{j\}) \sim (\A,S\setminus \{p_j\}, I\setminus \{j\}).$ Replacing in (\ref{eq:reduction}) we obtain 
\begin{equation}\label{eq:reduction2}
\pi_1(M(\A, S\setminus \{p_j\}, I\setminus \{j\}))/\left\langle\left\langle\gamma_{p_j} \right\rangle\right\rangle \cong \pi_1(M(\A,S,I)) 
\end{equation}
But $\gamma_{p_j}$ also generates the kernel of the map of fundamental group induced by the inclusion $$M(\A,S\setminus \{ p_j\},I\setminus\{ j\})\hookrightarrow M(\A\setminus L, S\setminus \{p_j\},I\setminus\{j\})$$
therefore
$$\pi_1(M(\A, S\setminus \{p_j\}, I\setminus \{j\}))/\left\langle\left\langle\gamma_{p_j} \right\rangle\right\rangle \cong \pi_1(M(\A\setminus \{L\},S\setminus \{p_j\}, I\setminus \{j\})) $$
which together with (\ref{eq:reduction2}) prove the statement.
\end{enumerate}
\end{proof}
 \begin{proof}[Proof of Theorem \ref{reducedLac} ]
 Given an arbitrary LAC datum $(\A, S, I)$ by Lemma \ref{RL1} we can suppose that $I\subset \{1,\ldots,s\}$. By Lemma \ref{RL2} all those points in $S$ not lying over $\A$ can be also discarded without changing the fundamental group. 
 
By Proposition \ref{RL3} 1, we remove from $S$ all points $p_j$ such that $j\not\in I$ so $S=I$, we will denote the LAC datum by $(A,S)$ .

 If there is a smooth point $p_j\in S$ such that $p_j\in L$ for some $L\in \A$ by Proposition \ref{RL3} 2, $(\A,S)\sim (\A\setminus \{L\}, S\setminus \{p_j\})$. This new LAC datum could have as well smooth points lying in $S\setminus \{p_j\}$, either coming from $S$ or from double points in $\A$ lying in $L$. We repeatedly apply Proposition \ref{RL3} 2, until $I\subset \Sing \A$ or $\A=\emptyset$. As there are only a number finite of points and lines this process must end and we obtain an equivalent reduced LAC datum $(\A',I')$ as wanted.

 \end{proof}
 
\subsection{A presentation for the orbifold fundamental group}
\begin{defn}
Let $\A=\{L_1,\ldots,L_k\}$, $\bar{X}$ the blow up of $\mathbb{P}^2$ at $\Sing\A=\{p_1,\ldots, p_s\}$  and $D_i$ as in section \ref{subsection4.2}. The divisor $D=\sum  D_i$ is SNC and for $r \in (\mathbb{N}^*\cup \{\infty\})^{k+s}$ the orbifold $\mathcal{X}(\bar{X},D,r)$ is called a  \emph{weighted LAC Surface}.
\end{defn}

\begin{thm}\label{MainT} Let $\A=\{L_1,\ldots,L_k\}$ be a real arrangement, $\mathcal{X}(\bar{X},D,r)$ a weighted LAC surface. Suppose we consider $L_k$ as a line at infinity and $\A^{\text{aff}}$ has no vertical line. Choose a base point $q$ and a canonical elementary geometric base $\Gamma^{(0)}=(\gamma_1,\ldots,\gamma_{k-1})$ based at $q$ and to the right of any vertex as in Section \ref{sec:FundG}. Let $\Sing \A=\{p_1,\ldots, p_s\}$ and $\gamma_{p_j}$ be the elementary singular meridian around $p_j$. Then the $\gamma_{p_j}$ can be expressed in terms of $\Gamma^{(0)}$ as in Lemmas \ref{finite} and \ref{meridian3} a presentation for $\pi_1(\mathcal{X}(\bar{X},D,r),q)$ is given by
\begin{equation}\label{eq:Presentation}
\pi_1(\mathcal{X}(\bar{X},D,r),q)=\left\langle \gamma_1,\ldots, \gamma_{k-1}\mid \bigcup_{p_l\in\Sing(\A(\mathbb{R}))}R_{p_l}, \gamma_i^{r_i},\gamma_{p_j}^{r_{k+j}}, \begin{array}{l}
i=1,\ldots, k ,\\
  j=1,\ldots, s 
\end{array}   \right\rangle \end{equation}
where we omit the relation $\gamma^{r_m}=1$ if $r_m=\infty$.
\end{thm}
\begin{proof}
We find first a presentation for $\pi_1(\bar{X}\setminus D,q)$ and express the meridians around the $D_r$ in terms of $\gamma_i$.
As $\bar{X}\setminus D\cong M(\A)$ by remark \ref{rema3}, we obtain that $\gamma_i$ is a meridian of $D_i$ in $\bar{X}$ and by Theorem \ref{presentationRandell} we have the following presentation for $\pi_1(\bar{X}\setminus D,q)$:
$$\pi_1(\bar{X}\setminus D,q)=\left\langle \gamma_1,\gamma_2,\ldots, \gamma_{k-1}\mid \bigcup_ {p\in\Sing(\A(\mathbb{R}))}R_p \right\rangle. $$

 The elementary meridian $\gamma_k$ around $D_k$ is given by Lemma \ref{lineatinfinity} as $\gamma_k=(\gamma_{k-1}\cdots \gamma_1)^{-1}$.    
 The meridians around the exceptional divisor $D_{k+j}$ are given by the Lemmas \ref{finite} and \ref{meridian3} in the following way: $\gamma_{p_j}$ is a meridian around $p_j$ lying completely in the line $\Sigma_i$, so after the blow up this meridian lies in the strict transform of $\Sigma_i$ giving a meridian of $D_{k+j}$. Moreover, $\gamma_{p_j}$ is expressed in terms of $\Gamma^{(0)}$. By \cite{Eyssidieux} p.3 dividing by the normal subgroup generated by $\gamma_i^{r_i},\gamma_{p_j}^{r_{k+j}}$ we obtain the presentation.
\end{proof}

\begin{cor}\label{corollary1} Let $(\A,I)$ be a reduced LAC surface with $\A$ real. A presentation for $\pi_1(\bar{X}\setminus (D_I)_\infty)$ is given by
\begin{equation}\label{eq:corollary1}
\pi_1(\bar{X}\setminus (D_I)_\infty)\cong\left\langle \gamma_1,\ldots, \gamma_{k-1}\mid \bigcup_{p_r\in\Sing(\A(\mathbb{R}))}R_{p_r},\gamma_{p_j}, \quad  j\in I  \right\rangle. \end{equation}
\end{cor}
\section{Applications}\label{sec:App}
\subsection{LAC Surface with infinity fundamental group and finite abelianization}\label{subsectionExample}
Consider a set $S$ of $4$ points in general position in $\mathbb{P}^2$. The arrangement $\mathcal{B}=\{L_1,\ldots, L_6\}$ of $6$ lines connecting each pair of these is called \emph{the complete cuadrilateral}. It has $4$ triple points and $3$ double points: $\Sing \mathcal{B}=\{p_1,\ldots, p_7\}$ numbered as in Figure \ref{fig:Ceva2}. It has the following equation  $(z_1^2-z_2^2)(z_1^2-z_3^2)(z_2^2-z_3^2)=0 $ for projective coordinates $(z_1:z_2:z_3)$.

If we consider $L_6$ as the line at infinity, after a small rotation in order to have no vertical lines, we obtain the real picture as in Fig. \ref{fig:Ceva2}.

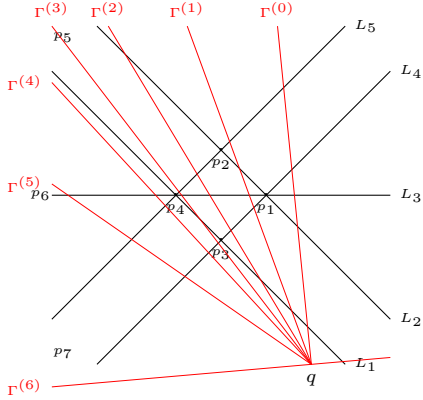
\begin{figure}[ht]
\centering
\begin{tikzpicture}[scale=1.5]
\draw (.8,-1.5) node [below] {\scriptsize $q$};

\node at (.4,0) [below]{\tiny $p_1 $};
\node at (0,.4)[below]{\tiny $p_2 $};
\node at (0,-.4)[below]{\tiny $p_3 $};
\node at (-.4,0)[below]{\tiny $p_4 $};
\node at (-1.4,1.4){\tiny $p_5 $};
\node at (-1.6,0){\tiny $p_6 $};
\node at (-1.4,-1.4){\tiny $p_7 $};

\draw (-1.5,0) -- (1.5,0) node[above,right]{\tiny $L_3$};
\draw (-1.5,-1.1) -- (1.1,1.5)node[above,right]{\tiny $L_5$};
\draw (-1.1,-1.5) -- (1.5,1.1)node[above,right]{\tiny $L_4$};
\draw (-1.5,1.1) -- (1.1,-1.5)node[above,right]{\tiny $L_1$};
\draw (-1.1,1.5) -- (1.5,-1.1)node[above,right]{\tiny $L_2$};

\draw [line width=0.05mm,red ] (.5,1.5)  node [above]{\tiny {$\Gamma^{(0)}$}} -- (.8,-1.5);
\draw [line width=0.05mm,red ] (-.3,1.5)node [above]{\tiny {$\Gamma^{(1)}$}} -- (.8,-1.5);
\draw [line width=0.05mm,red ] (.8,-1.5) -- (-1,1.5) node [above]{\tiny {$\Gamma^{(2)}$}};
\draw [line width=0.05mm,red ] (.8,-1.5) -- (-1.5,1.5) node [above]{\tiny {$\Gamma^{(3)}$}};
\draw [line width=0.05mm,red ] (.8,-1.5) -- (-1.5,1) node [left]{\tiny {$\Gamma^{(4)}$}};
\draw [line width=0.05mm,red ] (.8,-1.5) -- (-1.5,.1) node [left]{\tiny {$\Gamma^{(5)}$}};
\draw [line width=0.05mm,red ] (1.5,-331/230) -- (-1.5,-1.7) node [left]{\tiny {$\Gamma^{(6)}$}};

\foreach \Point in {(.4,0),(-.4,0),(0,.4),(0,-.4)}{
    \node at \Point {$\cdot$};
}


\end{tikzpicture}
\caption{Complete quadrilateral.}
\label{fig:Ceva2}
\end{figure}
\noindent By the subsections \ref{subsection2.2} and \ref{3} we have that the elementary geometric base (up to homotopy in $\pi_1(\mathbb{P}^2\setminus \mathcal{B},q)$ and replacing $\gamma_i^{(0)}$ by $x_i$) are
\begin{equation}\label{eq:ceva} 
\begin{aligned}
\Gamma^{(0)}&= (x_1,x_2,x_3,x_4,x_5)\\
\Gamma^{(1)}&= (x_1,x_4, x_3^{x_2},x_2,x_5)\\
\Gamma^{(2)}&= (x_1,x_4, x_3^{x_2},x_5,x_2)\\
\Gamma^{(3)}&= (x_4,x_1, x_3^{x_2},x_5,x_2)\\
\Gamma^{(4)}&= (x_4,x_5, x_3^{x_2x_1},x_1,x_2)\\
\Gamma^{(5)}&= (x_2^{x_3x_2x_1x_5x_4},x_1^{x_3^{x_2}x_1x_5x_4},x_4,x_5,x_3^{x_2x_1})\\
\Gamma^{(6)}&=(x_3^{a^{-1}x_2x_1a},x_2^{x_3x_2x_1x_5x_4},x_1^{x_3^ {x_2}x_1x_5x_4},x_4,x_5)
\end{aligned}
\end{equation}
where
$$a=(x_2x_1)^{x_3x_2x_1}x_5x_4 $$

By Theorem \ref{presentationRandell} we obtain the following presentation:
\begin{equation}\label{quadrilateralpres}
G=\pi_1(\mathbb{P}^2\setminus \mathcal{B},q)\cong\left\langle x_1, \ldots, x_5 \mid [x_4,x_1], [x_5,x_2], [x_4,x_3,x_2], [x_5, x_3^{x_2}, x_1] \right\rangle
\end{equation}

which can be easily seen to be a semidirect product $\mathbb{F}_2 \ltimes \mathbb{F}_3$ where $\mathbb{F}_2=\left\langle x_4, x_5 \right\rangle$ and $\mathbb{F}_3:=\left\langle x_1,x_2,x_3 \right\rangle $. 


 Let $\bar{X}$ denote the blow up of $\mathbb{P}^2$ at $\Sing \mathcal{B}$, to simplify denote $E_k=D_{6+k}$ the exceptional divisor coming from $p_k$.  Consider the reduced LAC surface $M(\mathcal{B},I)$ where $I$ consists of three triple points and two double ones. The simplest case is $I=\{p_1,p_2,p_3,p_4,p_5\}$.
\begin{thm}\label{Counterexample}
The reduced LAC surface $M(\mathcal{B},I)$ has infinite fundamental group and finite $H_1$.
\end{thm}
\begin{proof}
Consider the singular meridians $\gamma_{p_j}$ around $p_j$  for $j=1,\ldots, 5$, which by Lemmas \ref{finite} and \ref{meridian3} are given by 
\begin{equation}\label{eq:cevameridians}
\begin{array}{lll} 
 \gamma_{p_1} = x_4 x_3x_2, & 
\gamma_{p_3} = x_4 x_1,  & \gamma_{p_5} = x_3^{x_2 x_1}x_5 x_4,\\
\gamma_{p_2} = x_5 x_2, &
\gamma_{p_4} = x_5 x_3^{x_2} x_1. 
 &
\end{array} 
\end{equation}

\noindent By the corollary \ref{corollary1}  a presentation of $\pi_1(M(\mathcal{B},I))$ can be obtained by
$$H:=\pi_1(M(\mathcal{B},I))=\pi_1(\mathbb{P}^2\setminus \mathcal{B},q)/ \langle\langle \gamma_{p_1}, \gamma_{p_2}, \gamma_{p_3}, \gamma_{p_4}, \gamma_{p_5}  \rangle\rangle . $$
By making $\gamma_{p_2}=1$ and $\gamma_{p_3}=1$ we obtain $x_5=x_2^{-1}$ and $x_4=x_1^{-1}$, replacing them in (\ref{quadrilateralpres}) and (\ref{eq:cevameridians}), we obtain the presentation
$$H=\langle x_1,x_2,x_3| [x_1^{-1},x_3,x_2], [x_2^{-1},x_3^{x_2},x_1], x_1=x_3x_2, x_2=x_3^{x_2}x_1, x_3^{x_2x_1}=x_1x_2 \rangle $$ 
By replacing $x_1$ by $x_3x_2$ the relation $[x_2^{-1}x_3^{-1}, x_3,x_2]$ becomes trivial. So we are left with:
$$H=\langle x_2,x_3 \mid  [x_2^{-1},x_3^{x_2},x_3x_2], \  x_2=x_3^{x_2}x_3x_2, \  x_3^{x_2x_3x_2}=x_3x_2x_2 \rangle $$ 

By writing down the relations:
\begin{align}
x_2^{-2}(x_3x_2)^2=x_2^{-1}x_3x_2x_3=x_3x_2^{-1}x_3x_2\\
x_2^2=(x_3x_2)^2 \label{x22x3x2}\\
(x_3x_2)^2=x_2(x_3x_2)^2x_2\label{x3x3}
\end{align}

By replacing (\ref{x22x3x2}) in (\ref{x3x3}) we obtain that $x_2^2=1$, hence $(x_3x_2)^2=1$. Note that these two relations include all the precedent. Therefore we obtain the presentation 

$$H=\langle x_2, x_3 \mid x_2^2=1, (x_3x_2)^2=1 \rangle$$
which can be seen either as $\mathbb{Z}/2\mathbb{Z}*\mathbb{Z}/2\mathbb{Z}$ or as $\mathbb{Z}/2\mathbb{Z}\ltimes \mathbb{Z}$, by this we see that $H$ is infinite and its abelianization is finite.
\end{proof} 
We can clarify this example geometrically by means of the following proposition.

\begin{prop}
There exists an orbifold morphism from $M(\mathcal{B},I)$ to $\mathcal{X}(\mathbb{P}^1,D,r)$ where $D=[0:1]+[1:-1]+[1:0]$ and $r=(2,+\infty,2)$. The morphism comes from a pencil of conics and induces an isomorphism between orbifold fundamental groups.
\end{prop}
\begin{proof}

Consider a pencil $\mathscr{P}$ having $4$ fixed points in general position, which we may assume to be $S=\{p_1,p_4,p_5,p_7\}$. If we let $Q_1=(z_1^2-z_2^2),Q_2=(z_1^2-z_3^2)$ and $Q_3=(z_2^2-z_3^2)$ we have that the complete quadrilateral $\A$ is given by $Q=Q_1 Q_2 Q_3=0 $.

The pencil $\mathscr{P}$ can be written as 
$\mathscr{P}=aQ_1-bQ_2 $ with $a,b\in \mathbb{C}$ not both zero.
Note that $Q_3\in \mathscr{P}$ as 
$Q_3=Q_2 - Q_1$. This pencil defines a rational map
$$f_{\mathscr{P}}:\mathbb{P}^2\to \mathbb{P}^1, \ (z_1:z_2:z_3)\mapsto (Q_1(z_1:z_2:z_3),Q_2(z_1:z_2:z_3))  $$
whose indeterminacy locus is $S$. By blowing it up, we obtain a regular map $\tilde{f}:\Bl_S \mathbb{P}^2 \to \mathbb{P}^1$ with fiber over $(a:b)$ the strict transform of $aQ_1-bQ_2$. 

As any point lying in two elements of the pencil is a fixed point of it, for any $x\in \mathbb{P}^2\setminus S$ there is a unique curve $C\in\mathscr{P}$ passing through it. In particular for the double points $p_2\in \{z_1-z_2=0\}\cap\{z_1+z_2=0\}$ and $p_3\in \{z_1-z_3=0\}\cap\{z_1+z_3=0\}$ the curves are $Q_1$ and $Q_2$ respectively. This allows us to extend $\tilde{f}$ to the blow up of $\Bl_S\mathbb{P}^2$ at $p_2,p_3$ as
$f:\Bl_{S\cup\{p_2,p_3\}}\to \mathbb{P}^1.$
 We have that
$f(E_2)=(1:0)$ and $f(E_3)=(0:1)$. Let $X=\Bl_{S\cup\{p_2,p_3\}}\setminus \{Q \cup E_7 \}$. Note that $f|_X:X\to \mathbb{P}^1\setminus \{(1:1)\}$ as $f(Q_3)=(1,1)$.

Moreover $f|_X$ has double fibers in $(0:1)$ and $(1:0)$. For any other $(a:b)\in\mathbb{P}^1\setminus \{(1:1)\}$ the fiber is the strict transform of $aQ_1-bQ_2$ minus one point (corresponding to the intersection with $E_7$). The former assertion can be seen by local computations: 
Consider $\mathbb{P}^2$ and $\mathbb{P}^1$ with coordinates $(z_1:z_2:z_3)$ and $(u,v)$ respectively. Restricting to the standard open sets $W_3=\{z_3=1\}\subset \mathbb{P}^2$ and $V_2=\{v=1\}\subset \mathbb{P}^1$ we have that $\tilde{f}|_{W_3}=\frac{z_1^2-z_2^2}{z_1^2-1}$ with $z_1^2-1\not = 0$. Blowing up at $p_2=(0,0)$ and working in coordinates $(z_1,Z_2)$ (where $Z_2$ is the coordinate in $U_1\subset \mathbb{P}^1$)  we have that $f|_{W_3\cap U_1}= z_1^2 \frac{1-Z_2^2}{z_1^2-1}$. Analogous computations for the other open sets and for $p_3$ show that the fibers are indeed double.
The last part of the statement is then clear.
\end{proof}
There is a modification of Dimca's suggestion that may still hold.
\begin{ques} Let $X$ be a reduced LAC surface with $H_1(X)$ finite whose universal abelian cover has finite $H_1$. Is $\pi_1(X)$ finite?
\end{ques}

\subsection{Presentation for a weighted complete quadrilateral}
By considering weighted LAC surfaces $\mathcal{X}(\bar{X},D,r)$ we can study the ramified covers of $\bar{X}$  over $D$. In the case where all the lines of $D$  have the same weight Hirzebruch constructed a finite abelian cover in \cite{Hirzebruch}. If moreover we ask the cover to be a quotient of the ball, Deligne-Mostow have given weights (not necessarily equal) for this to hold \cite{DM}. 

Consider again the complete quadrilateral $\mathcal{B}=\{L_1,\ldots, L_6\}$ with the same notation as in \ref{subsectionExample}, suppose $L_6$ is the line at infinity. Let $\tilde{X}=\Bl_S \mathbb{P}^2\to \mathbb{P}^2$ be the blow up of $\mathbb{P}^2$ at the four triple points $S=\{p_1,p_4,p_5,p_7\}$ and $E_1, E_4, E_5, E_7$ be the respective exceptional divisors. 

 Consider the elementary geometric base $\Gamma^{(0)}=(x_1,\ldots, x_5)$. A meridian $x_6$ for the line at infinity around the point $\Sigma^{(0)}\cap L_6$ (recall that $\Sigma^{(0)}$ is the line where $\Gamma^{(0)}$ lies) is given by Lemma \ref{lineatinfinity}
\begin{equation}
x_6= (x_5x_4x_3x_2x_1)^{-1}.
\end{equation}
Denote by $\gamma_{p_i}$ the meridian around $E_i$. By  Lemma \ref{finite}, using respectively the elementary geometric bases $\Gamma^{(0)}$ and $\Gamma^{(3)}$ of (\ref{eq:ceva}), we obtain:
\begin{equation}
\begin{aligned}
y_1:=\gamma_{p_1}&=x_4x_3x_2 \\
y_2:=\gamma_{p_4}&= x_5 x_3^{x_2}x_1
\end{aligned}
\end{equation}

Finally, the meridians around the triple points lying in $L_6$ are given by Lemma \ref{meridian3} and bases $\Gamma^{(4)}$ and $\Gamma^{(6)}$ of (\ref{eq:ceva}).
\begin{equation}
\begin{aligned}
y_3:=\gamma_{p_5}&=x_3^{x_2x_1}x_5x_4 \\
y_4:=\gamma_{p_7}&=x_4^{-1}x_5^{-1}ax_3^{a^{-1}x_2x_1a}
\end{aligned}
\end{equation}
 
where $a=(x_2x_1)^{x_3x_2x_1}x_5x_4. $
\begin{prop} Let $\mathcal{B}$ be the complete quadrilateral, $\tilde{X}$, $\Gamma^{(0)}=(x_1,\ldots,x_5)$ and $y_i$ as above.
For any $r=(m_1,\ldots,m_4,n_1,\ldots, n_{6})\in (\mathbb{N}^*\cup{+\infty})^{10}$ as in \cite{Tretkoff} p.110, $D=E_1+E_4+E_5+E_7+\sum_{i=1}^{6} L_i$ we have a presentation for the fundamental group of the ball quotient $\mathcal{X}(\tilde{X},D,r)$ given by 
$$\pi_1(\mathcal{X}(\tilde{X},D,r))=\left\langle x_1,\ldots,  x_5 |  [x_4,x_1], [x_5,x_2], [x_4,x_3,x_2], [x_5, x_3^{x_2}, x_1] ,\\
 x_i^{n_i}, y_i^{m_i} \right\rangle $$
\end{prop}

\bibliographystyle{smfalpha}
\bibliography{sample}
Rodolfo Aguilar Aguilar \\
Université Grenoble-Alpes, Institut Fourier, 100 rue de Maths, 38402, Saint Martin d'Hères Cedex, France.\\
rodolfo.aguilar-aguilar@univ-grenoble-alpes.fr
\end{document}